\documentclass[12pt]{amsart}
\usepackage{dsfont}
\usepackage{mathrsfs} 
\usepackage{amssymb, changes}
\usepackage{amsmath}
\usepackage{amsfonts}
\usepackage{bbm}
\usepackage{amsthm}
\usepackage{amsbsy}
\usepackage{amsgen}
\usepackage{amscd}
\usepackage{amsopn}
\usepackage{amstext}
\usepackage{amsxtra}
\usepackage{times}
\usepackage[margin=60pt]{geometry}

\usepackage[colorlinks=true,linkcolor=blue,citecolor=blue,urlcolor=blue]{hyperref}
\usepackage[foot]{amsaddr}
\usepackage{comment}
\usepackage{xcolor}
\definecolor{lime}{HTML}{A6CE39}
\DeclareRobustCommand{\orcidicon}{%
	\begin{tikzpicture}
	\draw[lime, fill=lime] (0,0) 
	circle [radius=0.16] 
	node[white] {{\fontfamily{qag}\selectfont \tiny ID}};
	\draw[white, fill=white] (-0.0625,0.095) 
	circle [radius=0.007];
	\end{tikzpicture}
	\hspace{-2mm}
}
\foreach \x in {A, ..., Z}{%
	\expandafter\xdef\csname orcid\x\endcsname{\noexpand\href{https://orcid.org/\csname orcidauthor\x\endcsname}{\noexpand\orcidicon}}
}
\theoremstyle{plain}
\newtheorem{theorem}{Theorem}[section]
\newtheorem{lemma}[theorem]{Lemma}
\theoremstyle{remark}

\newcommand{\R}{\mathbb{R}}
\newcommand{\N}{\mathbb{N}}
\renewcommand{\P}{\mathbb{P}}

\newcommand{\E}{\mathbb{E}}
\newcommand{\F}{\mathcal{F}}
\newcommand{\1}{\mathds{1}}

\begin{document}

\title[Recurrence and transience of random walks with drift $\rho x^{\alpha}/t^{\beta}$]
{Recurrence and transience of random walks with drift $\rho x^{\alpha}/t^{\beta}$}

\author[]{Ngo P. N. Ngoc\orcidB{}}
\address[N. P. N. N.]{Institute of Research and Development, and  
Faculty of Natural Sciences, Duy Tan University, Danang 550000, Vietnam}

\email{ngopnguyenngoc@duytan.edu.vn}
\author[]{Tuan-Minh Nguyen\orcidA{}}
\address[T.-M. N.]{School of Mathematics, Monash University, 9 Rainforest Walk, Clayton 3800, Victoria, Australia}
\email{tuanminh.nguyen@monash.edu}


\begin{abstract}
Menshikov and Volkov [Electron. J. Probab. {\bf 13} (2008)] studied recurrence and transience of a class of Markovian random walks on $\R_+$ whose conditional drift depends on both time and position and is of order $\rho x^\alpha t^{-\beta}$ with $\rho>0$. The case on the critical line $2\beta-\alpha=1$, with $\alpha\in(-1,1)\setminus\{0\}$, remained open. We prove recurrence in this remaining case. Furthermore, we establish recurrence and transience criteria that complete the classification for $-1<\alpha<1$ and $\beta\ge 0$, without assuming the Markov property and under weaker assumptions on the increments than those imposed by Menshikov and Volkov.
\end{abstract}

\maketitle

\section{Introduction}\label{sec:intro}
\subsection{Model description and main results} 
Throughout this paper, we denote by $\mathbb Z_+$ and $\R_+$  the sets of non-negative integers and non-negative real numbers, respectively. Also, let $\N=\mathbb Z_+\setminus\{0\}$ be the set of natural numbers. 
Let $(X_t)_{t\in \mathbb Z_+}$ be a discrete-time $\R_+$-valued process adapted to a filtration
$(\F_t)_{t\in \mathbb Z_+}$, with deterministic initial state $X_0\in\R_+$. For each $t\in \N$, let
$\Delta_{t}:=X_{t}-X_{t-1}$ be the increment at time $t$. Denote $$\mu_t:=\E(\Delta_{t+1}\mid\F_t)\quad\text{for $t\in \mathbb Z_+$},$$ which stands for the conditional drift. Fix $a>0$. We say that the process $(X_t)_{t\in \mathbb Z_+}$ is:
\begin{itemize}
    \item \textit{recurrent} if $\P(X_t<a\, \text{infinitely often})=1$,
    \item \textit{transient} if $\P(X_t\to\infty)=1$.
\end{itemize}

Menshikov and Volkov \cite{MV} studied discrete-time Markov processes for which the conditional drift
is comparable to $\rho X_t^\alpha/t^{\beta}$ on $\{X_t>a\}$, where $\rho>0$, $\beta\ge\alpha$ and $\beta\ge 0$ are fixed parameters. Their model was motivated by a problem concerning Friedman-type urns posed by Freedman \cite{Freedman65}.
When $\beta\ge0$ and $\alpha<\min\{\beta,2\beta-1\}$ they proved the recurrence under uniform boundedness and non-degeneracy of the increments. When $0 \le \beta < 1$ and $2\beta - 1 < \alpha < \beta$, they proved transience under the same assumptions on the increments, together with an additional assumption on the time required to leave $[0,a]$. 

The case on the critical line $2\beta-\alpha=1$ with $-1<\alpha<1$ was left open in \cite[\S5.3]{MV},
apart from the point $(\alpha,\beta)=(0,1/2)$, where recurrence for every $\rho>0$ follows from a martingale law of the iterated logarithm; see also \cite[Proposition~2.7]{Freedman75}. The endpoint cases  $\alpha\in \{-1,1\}$ were treated separately.

When $\alpha=-1$ and $\beta=0$, the drift is comparable to $\rho/X_t$. This case corresponds to the classical Lamperti problem; see \cite{Lamperti60,Lamperti63} and \cite[Theorem~5]{MV}. Passage-time moments in this setting are studied in \cite{AIM}; a systematic treatment of near-critical processes by the Lyapunov function method is given in \cite{MPW}; and an application to polling systems appears in \cite{MZ}.

When $\alpha=\beta=1$, the drift is comparable to $\rho X_t/t$. In this case, the process exhibits a phase transition in the parameter $\rho$ between recurrence and transience, with critical value $\rho=1/2$; see \cite[Corollary~1]{MV}. On $\mathbb Z$, the model is related to the \textit{elephant random walk} introduced by Sch\"utz and Trimper \cite{SchutzTrimper}, where 
\[ \E(X_{t+1}-X_t\mid\F_t) = (2p-1)\frac{X_t}{t}\quad\text{for $t\in \N$.}\]
Thus $\rho=2p-1$, and the critical value $\rho=1/2$ corresponds to $p=3/4$. The recurrence and transience criteria for the multi-dimensional elephant random walks were established in \cite{Qin}. A notable extension of the Schütz–Trimper model to $\mathbb{R}$, in which the step distribution can be arbitrary, is the \emph{step-reinforced random walk}, introduced by Bertoin~\cite{B2020b,B2021,B2022,B2024}. The recurrence-transience phase transition of this model has been proved recently in \cite{Qin2025}.

A continuous-time diffusion counterpart of the Menshikov--Volkov model was later studied by Gradinaru and Offret~\cite{GO,O2014}. It describes a one-dimensional Brownian motion dynamics evolving in a time-dependent potential $V_{\rho,\alpha,\beta}$:
$$
{\rm d} X_t =  - \frac{1}{2} \partial_x V_{\rho,\alpha,\beta}(t, X_t)\, \mathrm{d} t + \mathrm{d} B_t, \quad X_{t_0} = x_0,
$$
where
$V_{\rho,\alpha,\beta}(t,x) := -\frac{2\rho}{\alpha+1}\,\frac{|x|^{\alpha+1}}{t^\beta}.$
More recently, limit laws for an extension of the Menshikov--Volkov model to $\R^d$ have been studied in \cite{NN} under stronger moment and conditional covariance assumptions. 

Abramov considered a time-inhomogeneous birth-and-death process in \cite{Abr1}. A subsequent paper \cite{Abr2} states a recurrence-transience criterion that is claimed to apply to time-inhomogeneous random walks with real-valued increments. We explain why the criterion in \cite{Abr2} does not provide a proof for recurrence and transience of the Menshikov-Volkov model. Theorem~2.1 of \cite{Abr2} is formulated in terms of the time-dependent rates $\lambda_{x,t}=\frac12+\phi(x,t)$ and $\mu_{x,t}=\frac12-\phi(x,t)$ which specify a birth-and-death mechanism with infinitesimal drift $2\phi(x,t)$.
Formally setting $\phi(x,t)=\frac{\rho}{2}\frac{x^\alpha}{t^\beta}$
gives $\phi(x,x^2)=\rho/(2x)$ when $2\beta=1+\alpha$. The criterion stated in \cite[Theorem~2.1]{Abr2} then gives recurrence for $\rho<1/2$ and transience for $\rho>1/2$. However, the process considered in \cite{Abr2} is a particular continuous-time compound Poisson process with prescribed jump-size distributions, whereas the conditional drift alone does not determine the transition law of a random walk with general real-valued increments. The proof in \cite{Abr2} does not establish an equivalence between these two models. There is also a direct inconsistency with a previously established special case. At $(\alpha,\beta)=(0,1/2)$, the formal specialization of \cite[Theorem~2.1]{Abr2} gives transience for $\rho>1/2$, whereas recurrence for every $\rho>0$ is proved in \cite[\S5.3]{MV} under the non-degeneracy and bounded increment assumptions imposed there. Thus the criterion stated in \cite{Abr2} cannot apply generally to processes covered by the model of \cite{MV}. 

Our aim in this paper is to establish recurrence and transience criteria for all $-1<\alpha<1$ and $\beta\ge 0$, without assuming the Markov property and under weaker assumptions on the increments than those imposed by Menshikov and Volkov. Throughout this paper, we assume the following two conditions for the square increments. 

\begin{itemize}
\item[(H1)] 
For every $\varepsilon>0$, there exists a deterministic $L<\infty$ such that, for every $t\geq0$, 
$$\E\left( \Delta_{t+1}^2\1_{\{|\Delta_{t+1}|>L\}} \mid\F_t \right) \leq\varepsilon\quad\text{a.s.}
$$

\item[(H2)] There exists
$B>0$ such that
\[
\E(\Delta_{t+1}^2\mid\F_t)\ge B
\quad\text{a.s. on }\{X_t\ge a\},\quad t\ge0.
\]
\end{itemize}
Our first main result below gives a unified recurrence theorem for $\beta\geq(1+\alpha)/2$ and $-1<\alpha<1$, and thus settles the equality case left open in \cite[\S5.3]{MV}.

\begin{theorem}\label{thm:main}
Let $\alpha\in(-1,1)$, $\beta\in \big[\frac{1+\alpha}{2},\infty\big)$ and $\rho>0$. Suppose that $(X_t)$ satisfies \textup{(H1)} and \textup{(H2)}
and that there exists a deterministic integer $t_0\ge1$ such that, for every
$t\ge t_0$,
\begin{equation}\label{eq:drift-upper}
\E(\Delta_{t+1}\mid\F_t)
\le
\rho\frac{X_t^\alpha}{t^\beta}
\quad\text{a.s. on }\{X_t\ge a\}.
\end{equation}
Then
\[
\P(X_t<a\ \text{infinitely often})=1.
\]
\end{theorem}

Theorem~\ref{thm:main} also weakens the structural assumptions imposed in \cite{MV}. Menshikov and Volkov assume uniformly bounded increments and formulate their recurrence theorem for Markov processes. Hypothesis \textup{(H1)} permits unbounded increments, and the proof below uses only conditional moment inequalities and stopping times, without requiring the Markov property.

Let $\sigma_0:=\zeta_0:=0$, and  for $j\geq0$, let 
$$\sigma_{j+1} := \inf\{t>\zeta_j:X_t\leq a\}, \quad \zeta_{j+1} := \inf\{t\geq\sigma_{j+1}:X_t>a\}$$
be the successive entrance and exit times of $[0,a]$. 
On $\{\sigma_j<\infty\}$, let 
$$R_j:=\zeta_j-\sigma_j$$
be the $j$-th sojourn duration of the process in $[0,a]$. To establish transience, in addition to \textup{(H1)} and \textup{(H2)}, we impose the following condition. 
\begin{itemize} \item[\textnormal{(H3)}]
There exists a deterministic constant $C<\infty$ such that for every $j\geq1$,
$$\E\left( R_j \mid\F_{\sigma_j} \right)\leq C \quad\text{a.s. on $\{\sigma_j<\infty\}$.}$$
\end{itemize}
Our following result for transience does not require Markov property. 
\begin{theorem}\label{thm:transience}
Let $\alpha\in(-1,1)$, $\beta\in \left[0, \frac{\alpha+1}{2}\right)$. 
Suppose that $(X_t)$ satisfies \textup{(H1), (H2), (H3)} and that there exist $\rho>0$ and a deterministic integer $t_0\geq1$ such that, for every $t\geq t_0$, \begin{equation}\label{ineq:drift2} \E(\Delta_{t+1}\mid\F_t) \geq \rho\frac{X_t^\alpha}{t^\beta} \quad \text{a.s. on $\{X_t\geq a\}$.}
\end{equation} 
Then $$ \P(X_t\to\infty)=1.$$
\end{theorem}

\subsection{Strategies of the proof} 

In our proof, we develop the recursive gambler's ruin approach used in \cite{MV}. However, we do not assume the Markov property and rely only on martingale and stopping time arguments. 

\textit{Recurrence.} Define the process
\[
Y_t:=\frac{X_t^2}{t},
\]
which was already used in the proof of Theorem~2 of \cite{MV}. Hypothesis \textup{(H1)} implies that the conditional second moments of the increments are uniformly bounded. When $2\beta\ge 1+\alpha$, we first notice that for all sufficiently large $t$, on $\{X_t\geq a\}$, 
\[\E(Y_{t+1}-Y_t\mid\F_t) \leq \frac{1}{t+1}\left( K+2\rho Y_t^{(1+\alpha)/2}-Y_t\right) \] 
for some deterministic constant $K<\infty$. Since $\alpha<1$, the right-hand side is negative for all sufficiently large $Y_t$. Using this supermartingale property, we show that there exist a deterministic constant $r<\infty$ such that, after every almost surely finite stopping time not smaller than $t_0$, there is almost surely a later time $s$ satisfying  $X_s^2\leq rs$. Fix such a time $s$. For a sufficiently large deterministic constant $N_0$, define $N:=\max\{N_0,2X_s\}$ and stop the process at 
\[ \tau:=\inf\{t>s:X_t<a\text{ or }X_t>N\}.\] 
The relation $2\beta\ge 1+\alpha$ implies that, for $\theta:=\min\{\alpha,0\}$, 
\[ \E(\Delta_{t+1}\mid\F_t) \leq cX_t^\theta N^{-(1+\theta)} \] 
on $\{s\leq t<\tau\}$, where $c$ is independent of $s$ and $N$.  Using this bound, we construct a supermartingale to show that $\tau<\infty$ almost surely and 
\[ \P(X_\tau<a\mid\F_s)\geq\nu\]
for a deterministic constant $\nu>0$ independent of $s$ and $N$. After $\tau$, the preceding estimate for $Y_t$ provides another finite time at which $X_t^2\leq rt$. Repeating this argument, we obtain almost surely finite stopping times $s_j\le \tau_j<s_{j+1}\le \tau_{j+1}$ such that the conditional probability of $\{X_{\tau_j}<a\}$  with respect to $\mathcal{F}_{s_j}$ is at least $\nu$. The probability that none of $k$ successive events occurs is at most $(1-\nu)^k$. Consequently, $X_t<a$ infinitely often almost surely.

\textit{Transience.} We first show that, after any almost surely finite stopping time $s\geq t_0$ with $X_s>a$ and for every sufficiently large $N$, the process has conditional probability at least $c/N$ of exceeding $N$ before returning to $[0,a]$. We also control the conditional expectation of the duration of this excursion. 

We next prove that there exists a deterministic constant $R<\infty$ such that
$$\P(X_t\to\infty\mid\F_s)\geq\frac12$$
whenever $X_s\geq R(1+s)^{\beta/(1+\alpha)}$. Starting from a position $x$, the accumulated positive drift over a suitable time interval is sufficient to make the process exceed $2x$, unless the centered increments have a large fluctuation. Hypothesis \textup{(H1)} controls these fluctuations. The condition $2\beta<1+\alpha$ makes their probabilities summable over successive doublings. Thus, with probability at least $1/2$, the process continues increasing and tends to infinity.

Let $D:=\{X_t\to\infty\}$. On $D^c$, the process must return to $[0,a]$ infinitely often. Otherwise, the preceding time estimate would give a sequence of stopping times $(u_k)_{k\ge 1}$ with $u_k\to \infty$ such that
$X_{u_j}>R(1+{u_j})^{\beta/(1+\alpha)}$. The preceding lower bound of the conditional probability of $D$ would then contradict that $\P(D\mid \mathcal{F}_{u_j})\to 0$ a.s. on $D^c$ by L\'evy's upward theorem.

After each return to $[0,a]$, we repeat the first argument, choosing
$N:=2R(1+s)^{\beta/(1+\alpha)}$,
where $s$ is the starting time of the new excursion. Hypothesis \textup{(H3)} controls the time spent in $[0,a]$, while the first estimate controls the time until the process either returns to $[0,a]$ or exceeds $N$. These bounds imply that the sum of the conditional probabilities of exceeding $N$ sufficiently quickly is infinite. By the conditional Borel--Cantelli, there exist infinitely many times $t$ at which
$X_t>R(1+t)^{\beta/(1+\alpha)}$.
At each such time, the conditional probability of $D$ is at least $1/2$, again contradicting L\'evy's upward theorem on $D^c$. Consequently, $X_t\to\infty$ almost surely.

\subsection*{AI disclosure} Microsoft 365 Copilot was used to assist with proofreading, technical computations, and refining the supermartingale argument in Step 2 of the proof of Lemma~\ref{lem:interval-exit}. The authors take full responsibility for all mathematical content.

\section{Recurrence}\label{sec:recurrence}

Recall that $Y_t=X_t^2/t$, for $t\ge1$. The first lemma asserts that there exists a deterministic constant $r$ sufficiently large such that, after every almost surely finite stopping time $s$, there is almost surely a later time $t\ge s$ satisfying  $Y_t\leq r$. Note that \textup{(H2)} is not required in this lemma.

\begin{lemma}\label{lem:Y-drift}
Assume $\alpha<1$, $\beta\ge \frac{1+\alpha}{2}$, $\rho>0$ and $a>0$. Suppose that \textup{(H1)} holds and the condition \eqref{eq:drift-upper} is satisfied a.s. on $\{X_t\ge a\}$ for every $t\ge t_0$ with some $t_0\ge 1$. 
Then there exists deterministic constant $r\ge 1$ such that for every almost surely finite stopping time $s\ge t_0$, we have a.s.
$$\sigma_s:=\inf\{t\ge s : Y_t\le r \}<\infty.$$
\end{lemma}

\begin{proof}
By (H1), there exists a sufficiently large deterministic $L_0<\infty$ such that, for each $t\ge 0$,
$$
\E\left[ \Delta_{t+1}^2\1_{\{|\Delta_{t+1}|>L_0\}} \mid \F_t \right]\leq1\quad\text{a.s.}
$$ Therefore, for every $t\geq0$, 
\begin{align}\label{cond.moment}
\E[\Delta_{t+1}^2\mid\F_t] &=
\E\left[ \Delta_{t+1}^2\1_{\{|\Delta_{t+1}|\leq L_0\}} \mid \F_t \right] + \E\left[ \Delta_{t+1}^2\1_{\{|\Delta_{t+1}|>L_0\}} \mid\F_t \right] \leq L_0^2+1\quad\text{a.s.}
\end{align}
Since $X_{t+1}=X_t+\Delta_{t+1}$, we have
\begin{equation}\label{eq:Y-drift}
\E(Y_{t+1}-Y_t\mid\F_t) =\frac{\E(\Delta_{t+1}^2\mid\F_t)
+2X_t\E(\Delta_{t+1}\mid\F_t)
-{X_t^2}/{t}}{t+1}.
\end{equation}
Since $\alpha<1$, we can choose $r\geq1$ sufficiently large such that $rt_0\ge a^2$ and
\begin{align}\label{r.choice}
    L_0^2+1+2\rho y^{(1+\alpha)/2}\leq\frac{y}{2}\quad\text{for all } y\geq r.
\end{align}
For $t\geq t_0$,  if $Y_t>r$, then $X_t=\sqrt{tY_t}\geq\sqrt{rt_0}\geq a$, so \eqref{eq:drift-upper} applies.
Combining \eqref{eq:Y-drift} together with \eqref{eq:drift-upper} and \eqref{cond.moment}, we thus have that on $\{Y_t>r\}$ with $t\geq t_0$,
\begin{align}  \nonumber
\E(Y_{t+1}-Y_t\mid\F_t)&\le \frac{L_0^2+1+2\rho X_t^{\alpha+1} t^{-\beta} - {X_t^2}/{t}}{t+1}
\\& \leq \frac{L_0^2+1+ 2\rho Y_t^{(1+\alpha)/2} -Y_t }{t+1}\leq -\frac{Y_t}{2(t+1)}\quad\text{a.s.},
\label{drift.Y2}
\end{align}
where in the second inequality we use $Y_t=X_t^2/t$ and $\beta\ge (1+\alpha)/2$, and in the last inequality we use \eqref{r.choice}.
First let $m\geq t_0$ be deterministic and define
$$
\sigma_m:=\inf\{t\geq m:Y_t\leq r\}\quad\text{and}\quad V_t:=Y_t \1_{\{\sigma_m>t\}}
\, \text{ for }t\geq m.
$$
On $\{\sigma_m\leq t\}$, we have $V_t=V_{t+1}=0$. On $\{\sigma_m>t\}$,
we have $V_{t+1} = Y_{t+1}\1_{\{Y_{t+1}>r\}} \leq Y_{t+1}.$
Using \eqref{drift.Y2}, we thus get
$$
\E(V_{t+1}\mid\F_t)\le \mathds{1}_{\{\sigma_m>t\}}\E( Y_{t+1}\mid \F_t )
\leq
\left(1-\frac{1}{2(t+1)}\right)V_t.
$$
Iterating the above inequality, we get
$$
\E(V_t\mid\F_m) \leq Y_m \prod_{k=m}^{t-1} \left(1-\frac{1}{2(k+1)}\right)\le C Y_m  \left(\frac{m}{t}\right)^{1/2}\quad\text{a.s.,}
$$
where $C<\infty$ is some deterministic constant.
On $\{\sigma_m>t\}$, we have $V_t\geq r$. Thus
$$
\P(\sigma_m>t\mid\F_m) \leq C \frac{Y_m}{r}\left(\frac{m}{t}\right)^{1/2}\quad\text{a.s.}$$
Taking $t\to\infty$, we obtain
$$
\P(\sigma_m=\infty\mid\F_m)=0\quad\text{a.s.}
$$

For an almost surely finite stopping time $s\ge t_0$, the same conclusion follows by applying the above argument on each event $\{s=m\}$ for  $m\geq t_0$. Therefore $\sigma_s<\infty$ almost surely.

\end{proof}

The next lemma concerns the time needed to exit an interval $[a,N]$ and the
probability of going below $a$ before exceeding $N$. This result extends \cite[Lemma~3]{MV} while the Markovian property is not required.

\begin{lemma}\label{lem:interval-exit}
Fix $\theta\in(-1,0]$, $c>0$, and $\gamma\in(0,1)$. Suppose that \textup{(H1)} and \textup{(H2)} hold. There exist deterministic constants $N_0<\infty$ and $\nu>0$ such that if $s$ is an almost surely finite stopping time and $N$ is an almost surely finite $\mathcal{F}_s$-measurable random variable satisfying the following conditions:
\begin{itemize}
    \item  $N\ge N_0$ and $X_s\in[a,\gamma N]$ a.s.,
    \item  for every deterministic $t\geq0$,
    \begin{equation}\label{eq:interval-drift}
\E(\Delta_{t+1}\mid\F_t)
\le
cX_t^\theta N^{-(1+\theta)}\quad\text{a.s. on $\{s\le t< \tau\}$,}
\end{equation}
with $\tau:=\inf\{u>s:X_u<a\ \text{or}\ X_u>N\}$,
\end{itemize} 
then a.s. 
$$\tau<\infty\quad \text{and} \quad\mathbb P(X_{\tau} <a\mid \mathcal{F}_s)\ge \nu.$$
\end{lemma}

\begin{proof}
Let $B$ be the constant in (H2). By (H1), there exists a deterministic $L\geq1$ such that, for every $t\geq0$, 
\begin{equation}\label{drift.upper}
    \mathbb E\left(\Delta_{t+1}^2\mathbf 1_{\{|\Delta_{t+1}|>L\}} \mid \mathcal F_t \right) \leq \frac{B}{2}
\end{equation} 
almost surely. By \eqref{drift.upper} and (H2), we have
that \begin{equation}
\label{drift.lower}
\mathbb E\left(\Delta_{t+1}^2\mathbf 1_{\{|\Delta_{t+1}|\leq L\}} \mid \mathcal F_t \right) \geq \frac{B}{2}\quad \text{a.s. on $\{X_t\geq a\}$}.\end{equation}
Define
\begin{equation}\label{eq:interval-constants}
\lambda:=\left(1+\frac{L}{a}\right)^\theta, \quad C:=\frac{8c}{(1+\theta)\lambda B}, \quad G(v):=\int_0^v e^{-Cw^{1+\theta}}\,{\rm d}w, \text{ for }v\geq0.
\end{equation}
and choose a deterministic $N_0\geq L$ sufficiently large such that \begin{equation}\label{eq:interval-threshold} C\left(\frac{L}{N_0}\right)^{1+\theta}\leq\log 2.\end{equation}
For each $n\geq N_0$, define 
$$g_n(y):=nG(y/n) \quad \text{for }y\geq0.$$ 
Then 
\begin{equation}\label{eq:g-derivatives}
g_n'(y)=e^{-C(y/n)^{1+\theta}}>0,\quad g_n''(y)=-C(1+\theta)y^\theta n^{-(1+\theta)}g_n'(y)\quad \text{for } y>0.
\end{equation}
Thus $g_n$ is increasing and concave, $g_n(0)=0$, and $0\leq g_n(y)\leq y$. Since $\theta\in(-1,0]$, the function $g_n''$ is integrable on every bounded interval in $(0,\infty)$, and $g_n'$ extends to an absolutely continuous function on every bounded interval in $[0,\infty)$. Define 
$$h_n(y):=g_n\bigl(y\wedge(n+L)\bigr) \quad\text{for }y\geq0.$$
The function $h_n$ is nonnegative, increasing, concave, and bounded by $g_n(n+L)$. 

\smallskip

\emph{Step 1.} Fix an integer $m\geq1$. For each integer $k\geq0$, let
$$Z_k^{(m)}:=\mathds{1}_{\{N\le m\}} h_N(X_{(s+k)\wedge\tau}).$$
In this step we prove that $(Z_k^{(m)})_{k\ge 0}$ is a bounded supermartingale with respect to $(\mathcal F_{s+k})_{k\geq0}$.

Note that $\{N\leq m\}\in\mathcal F_s$. On $\{N\leq m\}$, we have 
$$0\leq h_N(y)\leq g_N(N+L)\leq N+L\leq m+L.$$ 
Fix $k\geq0$. On the event $\{s+k<\tau\}$, we have $X_{s+k}<N+L$, the function $h_N$ is differentiable at $X_{s+k}$ and $h_N'(X_{s+k})=g_N'(X_{s+k})$. By the concavity of $h_N$, we thus have 
\begin{equation}
    \label{concave.inq} h_N(X_{s+k+1})-h_N(X_{s+k}) \leq g_N'(X_{s+k})\Delta_{s+k+1}\quad\text{on $\{s+k<\tau\}$}.
\end{equation}
On the event $\{s+k<\tau\}\cap\{|\Delta_{s+k+1}|\leq L\}$, we note that $0\leq X_{s+k}+v\Delta_{s+k+1}\leq N+L$ for every $v\in[0,1]$. Thus, on the same event, we have $h_N(X_{s+k}+v\Delta_{s+k+1})=g_N(X_{s+k}+v\Delta_{s+k+1})$ for every $v\in[0,1]$. By the absolute continuity of $g_N'$, we have that on $\{s+k<\tau\}\cap\{|\Delta_{s+k+1}|\leq L\}$,
\begin{equation}\label{eq:taylor-g}
g_N(X_{s+k+1})-g_N(X_{s+k}) = g_N'(X_{s+k})\Delta_{s+k+1} + \Delta_{s+k+1}^2\int_0^1(1-v)g_N''(X_{s+k}+v\Delta_{s+k+1})\,{\rm d}v.\end{equation}
By \eqref{eq:g-derivatives}, we note that $|g_N''(y)|$ is non-increasing for $y\in (0,\infty)$.
Since $X_{s+k}+v\Delta_{s+k+1}\leq X_{s+k}+L$ for every $v\in[0,1]$, we have $|g_N''(X_{s+k}+v\Delta_{s+k+1})|\geq|g_N''(X_{s+k}+L)|$ for a.e. $v\in[0,1]$. Since $g_N''\leq0$ on $(0,\infty)$, it follows from \eqref{eq:taylor-g} that on $\{s+k<\tau\}\cap\{|\Delta_{s+k+1}|\leq L\}$,
$$g_N(X_{s+k+1})-g_N(X_{s+k}) \leq g_N'(X_{s+k})\Delta_{s+k+1} - \frac12\Delta_{s+k+1}^2|g_N''(X_{s+k}+L)|.$$
Combining this inequality with \eqref{concave.inq}, we obtain that on $\{s+k<\tau\}$,
\begin{equation}\label{eq:taylor-upper}
h_N(X_{s+k+1})-h_N(X_{s+k}) \leq g_N'(X_{s+k})\Delta_{s+k+1} - \frac12\Delta_{s+k+1}^2\mathbf 1_{\{|\Delta_{s+k+1}|\leq L\}}|g_N''(X_{s+k}+L)|.
\end{equation}

For $x\geq a$ and $\theta\in(-1,0]$, we note that
$(x+L)^\theta
\geq\left(1+{L}/{a}\right)^\theta x^\theta = \lambda x^\theta$, $(x+L)^{1+\theta}\leq x^{1+\theta}+L^{1+\theta}$
and thus
\begin{align*}
g_N'(x+L) &= \exp\left[-C\left(\frac{x+L}{N}\right)^{1+\theta}\right]\geq g_N'(x)\exp\left[-C\left(\frac{L}{N}\right)^{1+\theta}\right]\geq \frac12g_N'(x),
\end{align*}
where the last inequality follows from \eqref{eq:interval-threshold}.
Hence, by \eqref{eq:g-derivatives},
\begin{equation}\label{eq:gsecond-comparison}
|g_N''(X_{s+k}+L)| \geq \frac12C(1+\theta)\lambda X_{s+k}^\theta N^{-(1+\theta)}g_N'(X_{s+k})\quad\text{on $\{s+k<\tau\}$}.
\end{equation}
Taking conditional expectations of \eqref{eq:taylor-upper} with respect to $\mathcal F_{s+k}$ and using \eqref{eq:interval-drift}, \eqref{drift.lower}, and
\eqref{eq:gsecond-comparison}, we obtain
\begin{align*} \mathbb E\left(h_N(X_{s+k+1})-h_N(X_{s+k})\mid\mathcal F_{s+k}\right) &\leq g_N'(X_{s+k})X_{s+k}^\theta N^{-(1+\theta)} \left[c-\frac18C(1+\theta)\lambda B\right] =0 \end{align*} almost surely on $\{N\leq m\}\cap\{s+k<\tau\}$.
It follows that, for each fixed $m\ge 1$,
$(Z_k^{(m)})_{k\geq0} $ is a bounded supermartingale with respect to $(\mathcal F_{s+k})_{k\geq0}$.

\smallskip

\emph{Step 2.} In this step we prove that $\tau<\infty$ almost surely. Define 
$$U_k^{(m)} :=\mathds{1}_{\{N\le m\}} g_N(N+L)-Z_k^{(m)}.$$
Then $(U_k^{(m)})_{k\geq0}$ is a nonnegative submartingale, and $0\leq U_k^{(m)}\le m+L$.
On $\{N\le m\}\cap\{s+k<\tau\}\cap\{|\Delta_{s+k+1}|\leq L\}$, we have
$0\leq X_{s+k}\leq N$ and $0\leq X_{s+k+1}\leq N+L.$ On the same event, $h_N$ is equal to $g_N$ at both $X_{s+k}$ and $X_{s+k+1}$. Therefore, by the mean-value theorem, we have
$$|U_{k+1}^{(m)}-U_k^{(m)}| = \mathds{1}_{\{N\le m\}} |g_N(X_{s+k+1})-g_N(X_{s+k})| \geq \mathds{1}_{\{N\le m\}} \left( \min_{0\leq z\leq N+L}g_N'(z) \right) |\Delta_{s+k+1}|$$
on $\{s+k<\tau\}\cap\{|\Delta_{s+k+1}|\leq L\}$. Since $g_N'$ is decreasing, $N\geq L$, and $\theta\leq0$, we notice that 
$\min_{0\leq z\leq N+L}g_N'(z) =g_N'(N+L) =e^{-C((N+L)/N)^{1+\theta}} \geq e^{-2C}$.
Therefore, by \eqref{drift.lower},
$$\mathbb E\left((U_{k+1}^{(m)}-U_k^{(m)})^2\mid\mathcal F_{s+k}\right) \geq \frac{B}{2}e^{-4C}\mathbf 1_{\{s+k<\tau,\, N\le m\}}.$$
Using the identity $(U_{k+1}^{(m)})^2-(U_k^{(m)})^2 =2U_k^{(m)}(U_{k+1}^{(m)}-U_k^{(m)})+(U_{k+1}^{(m)}-U_k^{(m)})^2$ and the fact that $(U_k^{(m)})_{k\ge 0}$ is a nonnegative submartingale, we thus get
\begin{equation}\label{eq:U-square-drift}
\mathbb E\big((U_{k+1}^{(m)})^2-(U_k^{(m)})^2\mid\mathcal F_{s+k}
\big) \geq \frac{B}{2}e^{-4C}\mathbf 1_{\{s+k<\tau,\, N\le m\}}.
\end{equation}
Summing \eqref{eq:U-square-drift} for $k=0,\ldots,n-1$, taking conditional expectation with respect to
$\F_{s}$ and using the fact that $U_n^{(m)}\le  m+ L$, we obtain
$$ \mathbb E\left(n\wedge(\tau-s)\mid\mathcal F_s\right) \leq \frac{2e^{4C}(m+L)^2}{B}\quad\text{on }\{N\leq m\}.$$
Taking $n\to\infty$ and applying the monotone convergence theorem, we have almost surely,
\begin{equation*} \mathbb E(\tau-s\mid\mathcal F_s) \leq  \frac{2e^{4C}(m+L)^2}{B} \quad\text{on }\{N\leq m\}.
\end{equation*}
Hence $\tau<\infty$ almost surely on $\{N\leq m\}$. Since $\{N<\infty\}= \bigcup_{m\geq1}\{N\leq m\}$ has probability one, $\tau<\infty$ almost surely.

\smallskip

\emph{Step 3.} In this step we show that 
\begin{equation}\label{eq:reach-lb}\mathbb P(X_\tau<a\mid\mathcal F_s) \geq \nu:=1-\frac{G(\gamma)}{G(1)}>0.\end{equation} 
Since $\tau<\infty$ almost surely, $Z_k^{(m)}\to \mathds{1}_{\{N\le m\}}h_N(X_\tau)$ almost surely as $k\to\infty$.
By the bounded convergence theorem and the supermartingale property, we have 
$$\mathbb E(h_N(X_\tau)\mid\mathcal F_s) \leq h_N(X_s)\quad\text{a.s. on }\{N\leq m\}.$$
Since $X_s\leq\gamma N$, we note that $h_N(X_s)=g_N(X_s)\leq g_N(\gamma N)=NG(\gamma)$. 
Furthermore, on $\{X_\tau>N\}$,  we have a.s.
$h_N(X_\tau)\geq h_N(N)=g_N(N)=NG(1).$
Hence, 
$$NG(1) \mathbb P(X_\tau>N\mid\mathcal F_s) \leq NG(\gamma)\quad\text{a.s. on }\{N\leq m\}.$$
Since $N\geq N_0>0$, we thus get that for each $m\ge 1$,
$$\mathbb P(X_\tau>N\mid\mathcal F_s) \leq\frac{G(\gamma)}{G(1)}\quad\text{ a.s. on }\{N\leq m\}.$$
Since $N<\infty$ a.s., we obtain a.s.
\begin{equation*}
     \mathbb P(X_\tau>N\mid\mathcal F_s) \leq \frac{G(\gamma)}{G(1)}.
\end{equation*}
Since $\tau=\inf\{t>s:X_t<a\text{ or }X_t>N\}$ is almost surely finite, the events $\{X_\tau<a\}$ and $\{X_\tau>N\}$ are disjoint and their union has probability one. Hence, \eqref{eq:reach-lb} is verified.
\end{proof}

\begin{proof}[Proof of Theorem~\ref{thm:main}]
Let $r$ be the constant given by Lemma~\ref{lem:Y-drift}.
Let $N_0$ and $\nu$ be constants as in Lemma~\ref{lem:interval-exit} applied with parameters $c:=\rho(4r)^\beta$, $\theta:=\min\{\alpha,0\}$, and $\gamma:=1/2$. Set $T_0= \max\{t_0, N_0^2\}$.

We define two sequences of stopping times $(s_j)_{j\ge 1}$ and $(\tau_j)_{j\ge 1}$ recursively as follows. We use the convention $\inf\varnothing=\infty$. Let
$$
s_1:=\inf\{t\geq T_0:Y_t\leq r\}.
$$
Suppose that $s_j$ has been defined. On $\{s_j=\infty\}$, set $\tau_j:=\infty$. On $\{s_j<\infty\}$, set
\begin{align}\nonumber
&N_j:=N_0,\quad \tau_j:=s_j \quad\text{on $\{X_{s_j}<a\}$};\\
\label{eq:exit-time}
&N_j:=\max\{N_0,2X_{s_j}\},\quad \tau_j:=\inf\{t>s_j:X_t<a\text{ or }X_t>N_j\}\quad\text{on $\{X_{s_j}\geq a\}$}
\end{align}
Finally,  set $$s_{j+1} := \begin{cases} \inf\{t\geq\tau_j+1:Y_t\leq r\}, &\text{if }\tau_j<\infty,\\ \infty, &\text{if }\tau_j=\infty. \end{cases}$$

We next prove by induction that both sequences of stopping times defined above are finite almost surely and estimate the probability of $X_{\tau_j}<a$ conditional on $\mathcal{F}_{s_j}$.  

If $Y_{T_0}\leq r$, then $s_1=T_0$. Otherwise, by Lemma~\ref{lem:Y-drift}, we have $s_1<\infty$ almost surely. Suppose that $s_j$ is a.s. finite. 
Define
\begin{equation}\label{eq:Hj-definition}
H_j:=\{\tau_j<\infty,\ X_{\tau_j}<a\}.\end{equation}
On the event $\{X_{s_j}<a\}$, we have a.s. $\tau_j=s_j<\infty$ and hence the event $H_j$ occurs with probability 1.
We now work on the case $X_{s_j}\geq a$. By the definition of $N_j$ in \eqref{eq:exit-time}, we note that
$X_{s_j}\in[a,N_j/2]$ on $\{X_{s_j}\geq a\}$. Moreover, since $Y_{s_j}\leq r$ and $s_j\geq T_0\geq N_0^2$, we have that a.s. on $\{X_{s_j}\geq a\}$,
\begin{align*}
\frac{N_j^2}{s_j}
&=\max\left\{\frac{N_0^2}{s_j},\frac{4X_{s_j}^2}{s_j}\right\}\leq
4r.
\end{align*}
Hence, for each deterministic $t\ge 0$,  we have a.s. on $\{X_{s_j}\geq a\}\cap\{s_j\leq t<\tau_j\}$,
\begin{equation}
    \label{eq:time-comparison} 
    \frac{N_j^2}{t}\leq \frac{N_j^2}{s_j}\leq 4r.
\end{equation}
Using \eqref{eq:drift-upper}, \eqref{eq:time-comparison}  and $2\beta\ge 1+\alpha$, we thus have
\begin{align*}
\E(\Delta_{t+1}\mid\F_t) &\leq \rho\frac{X_t^\alpha}{t^\beta}=\rho\left(\frac{N_j^2}{t}\right)^\beta X_t^\alpha N_j^{-2\beta}\leq cX_t^\alpha N_j^{-(1+\alpha)}
\end{align*}
almost surely on $\{X_{s_j}\geq a\}\cap\{s_j\leq t<\tau_j\}$. If $\alpha<0$, then $\theta=\alpha$, and the last expression is
$cX_t^\theta N_j^{-(1+\theta)}.$
If $\alpha\geq0$, then $\theta=0$, and $X_t\leq N_j$ for $t<\tau_j$, and hence
$$
X_t^\alpha N_j^{-(1+\alpha)}
\leq
N_j^{-1}
=
X_t^\theta N_j^{-(1+\theta)}.
$$
Thus, in either case, we obtain 
$$
\E(\Delta_{t+1}\mid\F_t) \leq cX_t^\theta N_j^{-(1+\theta)}\quad\text{a.s. on $\{X_{s_j}\geq a\}\cap\{s_j\leq t<\tau_j\}$.}
$$
Applying
Lemma~\ref{lem:interval-exit}, we have $\tau_j<\infty$ almost surely and
\begin{equation}\label{eq:Hj-bound}
\P(H_j\mid\F_{s_j})\geq\nu.
\end{equation}
If $Y_{\tau_j+1}\leq r$, then $s_{j+1}=\tau_j+1$. Otherwise, by Lemma~\ref{lem:Y-drift}, we have $s_{j+1}<\infty$ almost surely.
By induction, all $s_j$ and $\tau_j$ are finite almost surely. Moreover,
$$
s_j\leq\tau_j<s_{j+1}\leq\tau_{j+1}.
$$

For integers $J\geq1$ and $k\geq0$, note that $H_J,\ldots,H_{J+k-1}\in \mathcal F_{s_{J+k}}$. By \eqref{eq:Hj-bound}, we thus have
\begin{align*}
\P\Big(\bigcap_{j=J}^{J+k}H_j^c\Big) &= \E\left[ \1_{\bigcap_{j=J}^{J+k-1}H_j^c} \P(H_{J+k}^c\mid\F_{s_{J+k}}) \right] \leq (1-\nu) \P\Big(\bigcap_{j=J}^{J+k-1}H_j^c\Big).
\end{align*}
Iterating the above inequality, we obtain 
$\P\Big(\bigcap_{j=J}^{J+k}H_j^c\Big) \leq (1-\nu)^{k+1}$.
Taking $k\to\infty$ and using the continuity of probability measure, we get
$\P\Big(\bigcap_{j=J}^{\infty}H_j^c\Big)=0,$
yielding that $H_j$ occurs infinitely often almost surely. Hence
$$
\mathbb P(X_t<a\text{ infinitely often})=1.
$$
\end{proof}

\section{Transience}

\begin{lemma}\label{lem:one-trial} Assume that \textup{(H1)}, \textup{(H2)}, and
\eqref{ineq:drift2} hold. 
 There exist deterministic constants $N_0<\infty$, $c>0$ and $C<\infty$ such that the following holds. Let $s\geq t_0$ be an almost surely finite stopping time satisfying $X_s>a$, and let $N$ be an almost surely finite $\F_s$-measurable random variable satisfying $N\geq N_0$. Then there exist an almost surely finite stopping time $\tau\geq s$ and an event $G\in\F_\tau$ such that 
 \begin{align}
 &X_t>a\quad\text{for every }s\leq t<\tau, \label{eq:terminal}\\
& G\subseteq\{X_\tau>N\}, G^c\subseteq\{X_\tau\leq a\}, \label{eq:inclusion} \\
 \label{eq:trial-success} &\P(G\mid\F_s)\geq\frac{c}{N}, \\
  \label{eq:trial-duration} &\E(\tau-s\mid\F_s) \leq C\left(1+N^2\P(G\mid\F_s)\right)
  \end{align} almost surely.    
\end{lemma}

\begin{proof} 
By \textup{(H1)}, choose a deterministic $L\geq1$ sufficiently large that, for every $t\geq0$, $$\E\left( \Delta_{t+1}^2 \1_{\{|\Delta_{t+1}|>L\}} \mid\F_t \right) \leq \min\left\{ \frac{B}{4}, \frac{B}{4(a+1)} \right\}\quad\text{a.s.}$$
Set $$h:=\min\left\{\frac14,\frac{B}{4(L+a+1)}\right\},\quad p:=\frac{B}{4L(L+a+1)}.$$

\textit{Step 1.}  In this step we prove that
for every $t\geq t_0$, \begin{equation}\label{eq:upward-jump} \P(\Delta_{t+1}\geq h\mid\F_t)\geq p \quad\text{a.s. on $\{a<X_t<a+2h\}$}. \end{equation}

Fix $t\geq t_0$. By the choice of $L$ and \textup{(H2)}, $$\E\left( \Delta_{t+1}^2 \1_{\{|\Delta_{t+1}|\leq L\}} \mid\F_t \right) \geq \frac{3B}{4}\quad \text{a.s. on $\{X_t>a\}$.}$$
Moreover, by \eqref{ineq:drift2}, we note that a.s. on $\{X_t>a\}$,
$$\E\left( -\Delta_{t+1}\1_{\{\Delta_{t+1}<0\}} \mid\F_t \right) \leq \E\left( \Delta_{t+1}\1_{\{\Delta_{t+1}>0\}} \mid\F_t \right).$$
On $\{X_t<a+1\}$, we have $-\Delta_{t+1}\1_{\{\Delta_{t+1}<0\}}\leq X_t\leq a+1$, and thus
\begin{align*} \Delta_{t+1}^2\1_{\{|\Delta_{t+1}|\leq L\}} &\leq L\Delta_{t+1}\1_{\{0<\Delta_{t+1}\leq L\}} +(a+1)(-\Delta_{t+1})\1_{\{-L\leq\Delta_{t+1}<0\}}. \end{align*} 
Hence, a.s. on $\{a<X_t<a+1\}$,
\begin{align*}
\frac{3B}{4}
&\leq \E\left( \Delta_{t+1}^2\1_{\{|\Delta_{t+1}|\leq L\}} \mid\F_t \right)\\
&\leq (L+a+1)\E\left( \Delta_{t+1}\1_{\{0<\Delta_{t+1}\leq L\}} \mid\F_t \right)
+\frac{a+1}{L}\E\left( \Delta_{t+1}^2\1_{\{|\Delta_{t+1}|>L\}} \mid\F_t \right)\\
&\leq (L+a+1)\E\left( \Delta_{t+1}\1_{\{0<\Delta_{t+1}\leq L\}} \mid\F_t \right)+\frac{B}{4}.
\end{align*}  
Consequently, $$\E\left( \Delta_{t+1}\1_{\{0<\Delta_{t+1}\leq L\}} \mid\F_t \right) \geq \frac{B}{2(L+a+1)}\quad\text{a.s. on $\{a<X_t<a+1\}$.}$$  
Using the fact that $\Delta_{t+1}\1_{\{0<\Delta_{t+1}\leq L\}}\leq h+L\1_{\{\Delta_{t+1}\geq h\}},$ we thus obtain \eqref{eq:upward-jump}.

\textit{Step 2.} In this step we construct a stopping time $\tau$ and an event $G$ and show that $\tau$ is finite almost surely. 

Let $s\ge t_0$ be a finite stopping time satisfying $X_s>a$. 
Define $$\lambda:=\begin{cases} s, &\text{if }X_s\geq a+h,\\[1mm] \inf\{t\geq s:X_t\leq a\text{ or }X_t\geq a+h\}, &\text{if }a<X_s<a+h.\end{cases}$$
By \textup{(H1)}, choose a deterministic $N_0\geq a+2h$ sufficiently large that, for every $t\geq0$, 
\begin{equation}\label{eq:large-jump} \E\left( \Delta_{t+1}^2 \1_{\{|\Delta_{t+1}|>N_0/2\}} \mid\F_t \right) \leq \frac{B}{100}\quad\text{a.s.} \end{equation} 
Let $N$ be a $\F_s$-measurable random variable such that $N\ge N_0$ almost surely. Define 
$$\tau:=\begin{cases} \inf\{t\geq\lambda:X_t\leq a\text{ or }X_t>N\}, &\text{on $\{a+h\leq X_\lambda\leq N\}$},\\ \lambda, &\text{on $\{X_\lambda>N\}\cup\{X_\lambda<a+h\}$}.\end{cases}$$ 
Finally, define $G:=\{X_\lambda\geq a+h\}\cap\{\tau<\infty,\ X_\tau>N\}$.

We next prove that a.s.
\begin{equation}\label{eq:trial-launch} \P(X_\lambda\geq a+h\mid\F_s)\geq p \quad\text{and}\quad \E(\lambda-s\mid\F_s)\leq\frac{1}{p}.\end{equation}
Indeed, by repeated conditioning using \eqref{eq:upward-jump}, we have that on $\{a<X_s<a+h\}$,
$$\P(\lambda>s+n\mid\F_s) \leq (1-p)^n$$ almost surely for every integer $n\geq0$. 
In particular, $\lambda<\infty$ almost surely. Note that on $\{a<X_s<a+h\}$, if $\Delta_{s+1}\geq h$ then we must have $X_\lambda\geq a+h$. 
Consequently, \eqref{eq:trial-launch} holds a.s. on $\{a<X_s<a+h\}$. On $\{X_s\geq a+h\}$, these inequalities also hold since $\lambda=s$. Hence \eqref{eq:trial-launch} is verified.

We next work on $\{a+h\leq X_\lambda\leq N\}$. For every integer $n\geq1$, set $\tau_n:=\tau\wedge(\lambda+n)$. On $\{\lambda\leq t<\tau\}$, one has $X_t>a$. 
Therefore, a.s. on $\{\lambda\leq t<\tau\}$,
\begin{align*} \E(X_{t+1}^2-X_t^2\mid\F_t) &= 2X_t\E(\Delta_{t+1}\mid\F_t)+\E(\Delta_{t+1}^2\mid\F_t)\geq B. \end{align*}
Consequently, 
\begin{equation}\label{eq:square-lower} B\E(\tau_n-\lambda\mid\F_\lambda) \leq \E(X_{\tau_n}^2\mid\F_\lambda)-X_\lambda^2 \quad\text{a.s. on $\{a+h\leq X_\lambda\leq N\}$.}\end{equation} 
On $\{\tau>\lambda+n\}$, we have $X_{\tau_n}\leq N$. On $\{\tau\leq\lambda+n\}$, we have $X_{\tau-1}\leq N$. On the latter event, if $|\Delta_\tau|\leq N/2$, then $X_\tau\leq3N/2$ on $\{X_\tau>N\}$, while $X_\tau\leq a$ on $\{X_\tau\leq a\}$. On $\{\tau\leq\lambda+n,\ |\Delta_\tau|>N/2\}$, we have $X_\tau^2\leq2N^2+2\Delta_\tau^2\leq10\Delta_\tau^2$.
Therefore, a.s. on $\{a+h\leq X_\lambda\leq N\}$,
\begin{align} \E(X_{\tau_n}^2\mid\F_\lambda) \leq a^2&+\frac{9N^2}{4}\P(\tau<\infty,\ X_\tau>N\mid\F_\lambda)+N^2\P(\tau>\lambda+n\mid\F_\lambda)\notag\\ &+10\E\left( \Delta_\tau^2\1_{\{\tau\leq\lambda+n,\ |\Delta_\tau|>N/2\}} \mid\F_\lambda \right). \label{eq:square-upper} 
\end{align} 
Since $N$ is $\F_\lambda$-measurable and $N\geq N_0$, taking the decomposition over the possible values of $\tau$ and using \eqref{eq:large-jump}, we get that, a.s. on $\{a+h\leq X_\lambda\leq N\}$,
\begin{align*} 
\E\left(\Delta_\tau^2 \1_{\{\tau\leq\lambda+n,\ |\Delta_\tau|>N/2\}} \mid\F_\lambda \right) & \leq
\E\left[ \sum_{t=\lambda}^{\lambda+n-1} \1_{\{\tau>t\}} \E\left( \Delta_{t+1}^2 \1_{\{|\Delta_{t+1}|>N/2\}} \mid\F_t \right) \mid \F_\lambda \right]\\ &\leq\frac{B}{100}\E(\tau_n-\lambda\mid\F_\lambda).
\end{align*}
Combining \eqref{eq:square-lower} and \eqref{eq:square-upper}, and using the fact that $X_\lambda>a$, we obtain 
\begin{align}\label{eq:dur1}
\frac{9B}{10}\E(\tau_n-\lambda\mid\F_\lambda) &\leq \frac{9N^2}{4}\P(\tau<\infty,\ X_\tau>N\mid\F_\lambda)+N^2\P(\tau>\lambda+n\mid\F_\lambda)
\end{align}
almost surely on $\{a+h\leq X_\lambda\leq N\}$. 
Moreover, $$\P(\tau>\lambda+n\mid\F_\lambda) \leq \frac{1}{n}\E(\tau_n-\lambda\mid\F_\lambda)$$ almost surely on that event. 
For every integer $n\geq1$, on $\{a+h\leq X_\lambda\leq N,\ n\geq20N^2/(9B)\}$, the term $N^2\P(\tau>\lambda+n\mid\F_\lambda)$ in \eqref{eq:dur1} is thus at most $\frac{9B}{20}\E(\tau_n-\lambda\mid\F_\lambda)$.
Since $N<\infty$ almost surely, using this bound in \eqref{eq:dur1}, taking $n\to\infty$ and applying the conditional monotone convergence theorem, we obtain
\begin{equation}\label{eq:growth-duration} \E(\tau-\lambda\mid\F_\lambda) \leq \frac{5}{B}N^2\P(\tau<\infty,\ X_\tau>N\mid\F_\lambda)\quad \text{
a.s. on $\{a+h\leq X_\lambda\leq N\}$.}\end{equation} 
In particular, $\tau<\infty$ almost surely on that event. Since $\lambda<\infty$ almost surely and $\tau=\lambda$ on its complement, we have $\tau<\infty$ almost surely.

\textit{Step 3.} In this step, we complete the proof by establishing the lower bound for $\mathbb P(G\mid \F_s)$ and the upper bound for $\E(\tau-s\mid\mathcal{F}_s)$.

By \textup{(H1)}, $X_{\tau_n}-X_\lambda=\sum_{k=0}^{n-1}\1_{\{\lambda+k<\tau\}}\Delta_{\lambda+k+1}$ is integrable for every integer $n\geq1$. Taking conditional expectations and using \eqref{ineq:drift2}, we have $$X_\lambda\leq\E(X_{\tau_n}\mid\F_\lambda)\quad\text{a.s. on } \{a+h\leq X_\lambda\leq N\}.$$
On $\{\tau\leq\lambda+n,\ X_\tau>N,\ |\Delta_\tau|\leq N/2\}$, we have $X_\tau\leq3N/2$, while $X_\tau\leq a$ on $\{\tau\leq\lambda+n,\ X_\tau\leq a\}$. On $\{\tau\leq\lambda+n,\ |\Delta_\tau|>N/2\}$, we have $X_\tau\leq N+|\Delta_\tau|\leq3|\Delta_\tau|$.
Consequently, using $|z|\leq\frac{2}{N}z^2$ for $|z|>N/2$ and \eqref{eq:large-jump}, we obtain that,  a.s. on $\{a+h\leq X_\lambda\leq N\}$,
\begin{align*} \E(X_{\tau_n}\mid\F_\lambda) \leq a&+\frac{3N}{2}\P(X_\tau>N\mid\F_\lambda)+N\P(\tau>\lambda+n\mid\F_\lambda) \\ &+\frac{3B}{50N}\E(\tau_n-\lambda\mid\F_\lambda).
\end{align*}
Since $\tau<\infty$ almost surely, taking $n\to\infty$ and using \eqref{eq:growth-duration}, we obtain
$$h\leq X_\lambda-a\leq\left(\frac32+\frac{3}{10}\right)N\P(X_\tau>N\mid\F_\lambda)<2N\P(X_\tau>N\mid\F_\lambda)$$ almost surely on $\{a+h\leq X_\lambda\leq N\}$.
On $\{X_\lambda>N\}$, the event $G$ occurs.
Since $N\geq N_0\geq a+2h$, one has $h/(2N)\leq1$, and hence $$\P(G\mid\F_\lambda)\geq\frac{h}{2N}\1_{\{X_\lambda\geq a+h\}}$$ almost surely.
Taking conditional expectations with respect to $\F_s$ and using \eqref{eq:trial-launch}, we obtain $$\P(G\mid\F_s)\geq\frac{ph}{2N}$$ almost surely. Thus \eqref{eq:trial-success} holds with $c:=ph/2$.

Finally, $\tau-s=(\lambda-s)+\1_{\{X_\lambda\geq a+h\}}(\tau-\lambda)$.
Using \eqref{eq:trial-launch}, \eqref{eq:growth-duration}, and the tower property, we obtain
\begin{align*} \E(\tau-s\mid\F_s) &\leq \frac{1}{p}+\frac{5}{B}N^2 \E\left[\1_{\{X_\lambda\geq a+h\}}\P(X_\tau>N\mid\F_\lambda)\mid\F_s\right]\\ &=\frac{1}{p}+\frac{5}{B}N^2\P(G\mid\F_s) \end{align*} almost surely. This proves \eqref{eq:trial-duration}. By construction, $G^c\subseteq\{X_\tau\leq a\}$.
 \end{proof}

\begin{lemma}\label{lem:escape-scale} 
Let $\alpha\in(-1,1)$, $0\leq\beta<\frac{\alpha+1}{2}$. Assume that (H1) and \eqref{ineq:drift2} hold. Then there exists a deterministic constant $R<\infty$ such that, for every almost surely finite stopping time $s\geq t_0$, we have
$$\P(X_t\to\infty\mid\F_s)\geq \frac{1}{2}\quad\text{a.s. on $\big\{X_s\geq R(1+s)^{\frac{\beta}{1+\alpha}}\big\}$.}$$
\end{lemma}

\begin{proof} 
Fix a sufficiently large deterministic constant $R$. Let $s\geq t_0$ be an almost surely finite stopping time. We define the stopping times $(S_k)_{k\ge0}$, the random variables $(x_k)_{k\ge 0}$ and the events $(E_k)_{k\ge 0}$ recursively as follows. Set $S_0:=s$, $x_0:=X_s$ and $E_{-1}:=\{X_s\geq R(1+s)^{\frac{\beta}{1+\alpha}}\}$. Suppose that $S_k$, $x_k$ and $E_{k-1}$ have been defined.
Define $$H_k := \inf\left\{ m\geq1: \sum_{\ell=0}^{m-1} (1+S_k+\ell)^{-\beta} \geq \rho^{-1}2^{|\alpha|+2}x_k^{1-\alpha} \right\},$$ 
which is $\F_{S_k}$-measurable. Since $\beta<1$, one has $H_k<\infty$ almost surely. 
On $E_{k-1}$, define 
$$S_{k+1} := \inf\left\{ t\geq S_k: X_t\leq\frac{x_k}{2} \text{ or } X_t\geq2x_k \right\} \wedge(S_k+H_k), $$ and, on $E_{k-1}^c$, set $S_{k+1}:=S_k$. Set $$ x_{k+1}:=X_{S_{k+1}}, \quad E_k := E_{k-1} \cap \{X_{S_{k+1}}\geq2x_k\}.$$ 
Then every $S_k$ is an almost surely finite stopping time and \begin{equation}\label{eq:level-growth} x_{k+1}\geq2x_k \quad\text{a.s. on $E_k$}. \end{equation}
Choose $R\geq1$ such that $R(1+t_0)^{\frac{\beta}{1+\alpha}}>2a.$ 
On $\{X_s\geq R(1+s)^{\frac{\beta}{1+\alpha}}\}\cap E_{k-1}$, one has $x_k\geq x_0$, and hence $\frac{x_k}{2}>a.$ 
Therefore, $\frac{x_k}{2}<X_t<2x_k$ almost surely on $\{S_k\leq t<S_{k+1}\}\cap \{X_s\geq R(1+s)^{\frac{\beta}{1+\alpha}}\}\cap E_{k-1}.$ 
Hence $X_t^\alpha\geq2^{-|\alpha|}x_k^\alpha$ on the same event. 
It follows from \eqref{ineq:drift2} that \begin{equation}\label{eq:escape-drift} \E(\Delta_{t+1}\mid\F_t) \geq \rho2^{-|\alpha|}x_k^\alpha(1+t)^{-\beta} 
\end{equation} 
almost surely on $\{S_k\leq t<S_{k+1}\}\cap \{X_s\geq R(1+s)^{\frac{\beta}{1+\alpha}}\}\cap E_{k-1}.$

For $n\geq0$, define $$ M_n^{(k)} := \sum_{\ell=0}^{n-1} \1_{\{S_k+\ell<S_{k+1}\}} \left( \Delta_{S_k+\ell+1} - \E(\Delta_{S_k+\ell+1}\mid\F_{S_k+\ell}) \right). $$
Then $(M_n^{(k)})_{n\geq0}$ is a martingale with respect to $(\F_{S_k+n})_{n\geq0}$. By \textup{(H1)}, there exists a deterministic constant $K_0<\infty$ such that, a.s.,
$$ \E\left( \left( M_{n+1}^{(k)}-M_n^{(k)} \right)^2 \mid\F_{S_k+n} \right) \leq K_0\1_{\{S_k+n<S_{k+1}\}}. $$

On $E_{k-1}\setminus E_k$, either $X_{S_{k+1}}\leq x_k/2$, or $S_{k+1}=S_k+H_k$ and $X_{S_{k+1}}<2x_k$. 
In the first case, since the conditional drift is nonnegative before $S_{k+1}$, 
$$ M_{S_{k+1}-S_k}^{(k)} \leq -\frac{x_k}{2}\quad\text{a.s. on $\{X_s\geq R(1+s)^{\frac{\beta}{1+\alpha}}\}\cap E_{k-1}$.} $$ 
In the second case, by \eqref{eq:escape-drift} and the definition of $H_k$, we have 
$$ \sum_{\ell=0}^{H_k-1} \E(\Delta_{S_k+\ell+1}\mid\F_{S_k+\ell}) \geq \rho2^{-|\alpha|}x_k^\alpha \sum_{\ell=0}^{H_k-1} (1+S_k+\ell)^{-\beta} \geq \rho2^{-|\alpha|}x_k^{\alpha}\cdot \rho^{-1}2^{|\alpha|+2}x_k^{1-\alpha} =4x_k $$ almost surely on $\{X_s\geq R(1+s)^{\frac{\beta}{1+\alpha}}\}\cap E_{k-1}.$ 
Since $X_{S_k+H_k}<2x_k$, one has $$ M_{H_k}^{(k)}\leq-3x_k $$ almost surely on the same event. 
Consequently, $$ E_{k-1}\setminus E_k \subseteq \left\{ \max_{0\leq n\leq H_k} |M_n^{(k)}| \geq \frac{x_k}{2} \right\} \quad\text{a.s. on $\{X_s\geq R(1+s)^{\frac{\beta}{1+\alpha}}\}\cap E_{k-1}.$}$$  
Since $H_k$ is $\F_{S_k}$-measurable, applying the conditional Doob inequality on each event $\{H_k=m\}$ and summing over $m\geq1$, we obtain
\begin{equation}\label{eq:escape-failure} \P(E_{k-1}\setminus E_k\mid\F_{S_k}) \leq \frac{4K_0H_k}{x_k^2} \quad\text{a.s. on $\{X_s\geq R(1+s)^{\frac{\beta}{1+\alpha}}\}\cap E_{k-1}.$}\end{equation} 

We next estimate $H_k$. By the definition and minimality of $H_k$, $$ \sum_{\ell=0}^{H_k-1}(1+S_k+\ell)^{-\beta}\leq \rho^{-1}2^{|\alpha|+2}x_k^{1-\alpha}+1. $$ By integral comparison, there exists a deterministic constant $C<\infty$ such that 
\begin{equation}\label{eq:escape-bound} H_k\leq C\left((1+S_k)^\beta x_k^{1-\alpha}+x_k^{\frac{1-\alpha}{1-\beta}}+1\right). \end{equation} almost surely on $\{X_s\geq R(1+s)^{\frac{\beta}{1+\alpha}}\}\cap E_{k-1}$.
On $E_i$, one has $S_{i+1}\leq S_i+H_i$.
Moreover, 
\begin{align*} (1+S_{i+1})^{1-\beta}-(1+S_i)^{1-\beta} &\leq (1-\beta)\sum_{\ell=0}^{H_i-1}(1+S_i+\ell)^{-\beta}\leq C\left(x_i^{1-\alpha}+1\right) \end{align*} 
almost surely on $\{X_s\geq R(1+s)^{\frac{\beta}{1+\alpha}}\}\cap E_i$. 
Since $R\geq1$, one has $x_i\geq1$ on $E_i$. 
Therefore, $$ (1+S_{i+1})^{1-\beta}-(1+S_i)^{1-\beta}\leq Cx_i^{1-\alpha}\quad\text{a.s. on $\{X_s\geq R(1+s)^{\frac{\beta}{1+\alpha}}\}\cap E_i$.}$$  
Summing over $i=0,\ldots,k-1$ and using \eqref{eq:level-growth}, we obtain \begin{align*} (1+S_k)^{1-\beta} &\leq (1+s)^{1-\beta}+C\sum_{i=0}^{k-1}x_i^{1-\alpha}\leq C\left((1+s)^{1-\beta}+x_k^{1-\alpha}\right) \end{align*} almost surely on $\{X_s\geq R(1+s)^{\frac{\beta}{1+\alpha}}\}\cap E_{k-1}$. 
Consequently, $$ (1+S_k)^\beta\leq C\left((1+s)^\beta+x_k^{\frac{\beta(1-\alpha)}{1-\beta}}\right) $$ almost surely on the same event. 
Combining this estimate with \eqref{eq:escape-bound} and \eqref{eq:escape-failure}, we obtain 
\begin{align*} \P(E_{k-1}\setminus E_k\mid\F_{S_k}) &\leq C\left((1+s)^\beta x_k^{-1-\alpha}+x_k^{-\frac{1+\alpha-2\beta}{1-\beta}}+x_k^{-2}\right) \end{align*} almost surely on $\{X_s\geq R(1+s)^{\frac{\beta}{1+\alpha}}\}\cap E_{k-1}$. 
By \eqref{eq:level-growth}, on $E_{k-1}$, we have $x_k\geq2^kx_0$. Since $E_{k-1}\in\F_{S_k}$, we have $\P(E_{k-1}\setminus E_k\mid\F_{S_k})=0$ on $E_{k-1}^c$. Taking conditional expectations with respect to $\F_s$, summing the geometric series and choosing sufficiently large $C$, we obtain 
\begin{align*} \sum_{k=0}^{\infty}\P(E_{k-1}\setminus E_k\mid\F_s) &\leq C\left((1+s)^\beta x_0^{-1-\alpha}+x_0^{-\frac{1+\alpha-2\beta}{1-\beta}}+x_0^{-2}\right) \end{align*}
almost surely on $\{X_s\geq R(1+s)^{\frac{\beta}{1+\alpha}}\}$.
Since $x_0\geq R(1+s)^{\frac{\beta}{1+\alpha}}$ on $E_{-1}$, one has $(1+s)^\beta x_0^{-1-\alpha}\leq R^{-1-\alpha}$. The other two terms are bounded by the corresponding negative powers of $R$. 
Choose $R$ sufficiently large that $$C\left(R^{-1-\alpha}+R^{-\frac{1+\alpha-2\beta}{1-\beta}}+R^{-2}\right)\leq\frac12.$$ 
For every integer $m\geq0$, we thus get
\begin{align*} \P(E_m^c\mid\F_s) &=\P\left(\bigcup_{k=0}^m(E_{k-1}\setminus E_k)\mid\F_s\right)\leq\sum_{k=0}^m\E\left[\1_{E_{k-1}}\P(E_{k-1}\setminus E_k\mid\F_{S_k})\mid\F_s\right]\leq\frac12 \end{align*} almost surely on $\{X_s\geq R(1+s)^{\frac{\beta}{1+\alpha}}\}$. 
Taking $m\to\infty$, we obtain $$ \P\left(\bigcap_{k=0}^{\infty}E_k\mid\F_s\right)\geq\frac12\quad\text{a.s. on $\{X_s\geq R(1+s)^{\frac{\beta}{1+\alpha}}\}$.} $$ On $\bigcap_{k=0}^{\infty}E_k$, every $S_k$ is finite and $x_k\geq2^kx_0\to\infty$.
Moreover, $X_t>x_k/2$ for every $S_k\leq t<S_{k+1}$. Since $S_{k+1}>S_k$, one has $S_k\to\infty$. Therefore, 
$$ X_t\to\infty\quad\text{a.s. on $\bigcap_{k=0}^{\infty}E_k$}. $$ 
Consequently, 
$$ \P(X_t\to\infty\mid\F_s)\geq\frac12\quad\text{a.s. on $\{X_s\geq R(1+s)^{\frac{\beta}{1+\alpha}}\}$.} $$ 
\end{proof}

\begin{proof}[Proof of Theorem~\ref{thm:transience}]
By \textup{(H3)}, there exists a deterministic constant $C<\infty$ such that \begin{equation}\label{eq:sojourn-mean} \E(R_j\mid\F_{\sigma_j})\leq C \end{equation} almost surely on $\{\sigma_j<\infty\}$ for every $j\geq1$.
Let $$ D:=\{X_t\to\infty\}.$$ Let $R$ be the deterministic constant given by Lemma~\ref{lem:escape-scale}. Choose $R$ sufficiently large such that $R\geq N_0$, where $N_0$ is the deterministic constant in Lemma~\ref{lem:one-trial}.

\textit{Step 1.} In this step, we show that
\begin{equation}\label{eq:cond-rt} X_t\leq a \quad\text{infinitely often a.s. on $D^c$}.
\end{equation}

Choose a sufficiently large deterministic integer $T_0\ge t_0$. Let $u\geq T_0$ be an almost surely finite stopping time. For every integer $n\geq1$, set
$$ N:=R(1+u+n)^{\frac{\beta}{1+\alpha}}.$$
The random variable $N$ is $\F_u$-measurable and satisfies $N\geq N_0$. On $\{X_u>N\}$, it is clear that $X_u>R(1+u)^{\frac{\beta}{1+\alpha}}$. Applying Lemma~\ref{lem:one-trial} on $\{a<X_u\leq N\}$ and setting $\tau:=u$ on its complement, we obtain an almost surely finite stopping time $\tau\geq u$ such that
$$\E(\tau-u\mid\F_u) \leq C( 1+N^2 ) \leq C\left( 1+R^2(1+u+n)^{\frac{2\beta}{1+\alpha}}\right)\quad\text{a.s. on $\{a < X_u\leq N\}$.}$$
Note that $\{a<X_t\leq R(1+t)^{\frac{\beta}{1+\alpha}} \text{ for every }u\leq t\leq u+n\}\subset \{\tau-u>n\}$.
Therefore, a.s. on $\{X_u>a\}$,
\begin{align*} &\P\left( a<X_t\leq R(1+t)^{\frac{\beta}{1+\alpha}} \text{ for every }u\leq t\leq u+n \mid \F_u \right)\leq \frac{C}{n} \left( 1+R^2(1+u+n)^{{\frac{2\beta}{1+\alpha}}} \right).
\end{align*}
Since $0\le \frac{\beta}{\alpha+1}<1/2$, the right-hand side converges to zero as $n\to\infty$. Consequently, almost surely on $\{X_u>a\}$, there exists a later time $t\ge u$ such that either $X_t\le a$ or $X_t>R(1+t)^{\frac{\beta}{1+\alpha}}$.
Fix a deterministic integer $m\geq T_0$. Set $u_0:=m$, and recursively define
$$ u_{k+1}:=\inf\left\{t\geq u_k+1:X_t\leq a\text{ or }X_t>R(1+t)^{\frac{\beta}{1+\alpha}}\right\},$$ with the convention $\inf\varnothing=\infty$. Applying this conclusion at $u=u_k+1$ on $\{u_k<\infty\}$, we obtain $u_k<\infty$ for every $k\geq1$ almost surely on $\{X_t>a\text{ for every }t\geq m\}$. Hence on the same event, $$X_{u_k}>R(1+u_k)^{\frac{\beta}{1+\alpha}} \quad\text{for each $k\geq1$}\quad\text{and}\quad u_k\to\infty.$$
Applying Lemma~\ref{lem:escape-scale}, we thus have $$\P(D\mid\F_{u_k})\geq 1/2 \quad\text{for each $k\ge 1$ a.s. on $\{X_t>a\text{ for every }t\geq m\}$}.$$ By L\'evy's upward theorem, we have a.s. $\P(D\mid\F_{u_k})\rightarrow \1_D=0$ on $D^c\cap\{X_t>a\text{ for every }t\geq m\}$ as $k\to\infty$, which is a contradiction. Hence $$ \P\left(D^c\cap\{X_t>a\text{ for every }t\geq m\}\right)=0. $$ Taking the union over all deterministic integers $m\geq T_0$, we obtain \eqref{eq:cond-rt}.

\textit{Step 2.} Set $\tau_0:=T_0$ and $S_0:=\inf\{r\geq \tau_0:X_r\leq a\}$. In this step we construct stopping times $(\tau_j)_{j\ge 1}$ and events $(G_j)_{j\ge 1}$, with $G_j\in\F_{\tau_j}$, such that for every $j\geq1$, the stopping times
$$S_j:=\inf\{r\geq\tau_j:X_r\leq a\}, \quad T_{j}:= \inf\left\{ t\geq S_{j-1}: X_t>a \right\}$$
and $\tau_j$ are finite almost surely on $D^c$, and $\tau_j\geq T_j$.
We set $\tau_j:=\infty$ on $\{T_j=\infty\}$ and take $G_j\subseteq\{T_j<\infty\}$. We interpret $\tau_j-T_j$ as zero on $\{T_j=\infty\}$ and $T_{j+1}-S_j$ as zero on $\{S_j=\infty\}$.
Moreover, $\tau_j<\infty$ almost surely on $\{T_j<\infty\}$, and the following relations hold almost surely on this event:
\begin{align}
\label{eq:inclusion-transience}
&G_j\subset \{X_{\tau_j}> 2R(1+T_j)^{\frac{\beta}{1+\alpha}} \},\quad G_j^c\subset \{X_{\tau_j}\le a\},\\
\label{eq:trial-prob} & \P(G_j\mid\F_{T_j})\geq c(1+T_j)^{-{\frac{\beta}{1+\alpha}}},\\
\label{eq:trial-dur}
& \E(\tau_j-T_j\mid\F_{T_j}) \leq C\left( 1+(1+T_j)^{{\frac{2\beta}{1+\alpha}}}\P(G_j\mid\F_{T_j}) \right).
\end{align}

By \eqref{eq:cond-rt}, we have $S_0<\infty$ a.s. on $D^c$. By \eqref{eq:sojourn-mean}, every sojourn beginning at a finite $\sigma_i$ ends almost surely, and hence $T_1<\infty$ a.s. on $D^c$. Suppose that $T_j<\infty$ a.s. on $D^c$. Applying Lemma~\ref{lem:one-trial} at time $T_j$ with $N:=2R(1+T_j)^{\frac{\beta}{1+\alpha}}$ on $\{T_j<\infty\}$, we obtain an event $G_j\in\F_{\tau_j}$ and a stopping time $\tau_j\geq T_j$, finite almost surely on $\{T_j<\infty\}$, such that \eqref{eq:inclusion-transience}, \eqref{eq:trial-prob} and \eqref{eq:trial-dur} hold on this event.
By \eqref{eq:terminal}, $X_t>a$ for every $T_j\leq t<\tau_j$ on $\{T_j<\infty\}$, and hence $S_j=\inf\{t>T_j:X_t\leq a\}$ on this event. On $\{S_j=\sigma_i<\infty\}$, we have $T_{j+1}=\zeta_i$. Since $\{S_j=\sigma_i<\infty\}\in\F_{\sigma_i}$, it follows from \eqref{eq:sojourn-mean} that
\begin{equation}\label{eq:wait}
\E(T_{j+1}-S_j\mid\F_{S_j})\leq C\quad\text{a.s. on $\{S_j<\infty\}$.}
\end{equation}
By \eqref{eq:cond-rt}, we must have $S_j<\infty$ a.s. on $D^c$. By \eqref{eq:wait}, we thus obtain that $T_{j+1}<\infty$ a.s. on $D^c$. Applying this argument recursively, we obtain the claim of Step 2.

\textit{Step 3.} In this step, we prove that
\begin{equation}\label{eq:diverge} \sum_{j=1}^{\infty}\P(G_j\mid\F_{T_j})=\infty \quad\text{a.s. on $D^c$}.\end{equation}
Set $$ \mathcal E := D^c \cap \left\{ \sum_{j=1}^{\infty}\P(G_j\mid\F_{T_j})<\infty \right\}. $$ We prove that $\P(\mathcal E)=0$.
Since $G_j\in\F_{\tau_j}\subseteq\F_{T_{j+1}}$, by the conditional Borel--Cantelli lemma, $G_j^c$ occurs for every sufficiently large $j$ almost surely on $\mathcal E$.
By \eqref{eq:inclusion-transience}, we then have $X_{\tau_j}\leq a$ for every sufficiently large $j$ a.s. on $\mathcal E$. Hence
\begin{align}\label{eq:Stau}
S_j=\tau_j \quad\text{for every sufficiently large $j$ a.s. on $\mathcal E$}.\end{align}
If $\beta=0$, then by \eqref{eq:trial-prob},
$$ \P(G_j\mid\F_{T_j})\geq c$$
for every $j$ a.s. on $\mathcal E$, which contradicts the definition of $\mathcal E$. We may therefore assume that $\beta>0$. We set $(1+T_j)^{-2\beta/(1+\alpha)}:=0$ on $\{T_j=\infty\}$. Since $G_j\subseteq\{T_j<\infty\}$, we have $\P(G_j\mid\F_{T_j})=0$ on $\{T_j=\infty\}$. By \eqref{eq:trial-prob}, we thus have for every $j\geq1$,
$$ (1+T_j)^{-\frac{2\beta}{1+\alpha}} \leq C\P(G_j\mid\F_{T_j})^2 \leq C\P(G_j\mid\F_{T_j})\quad\text{a.s.}$$
For every deterministic integer $m\geq1$, define $$ \nu_m := \inf\left\{ n\geq1: \sum_{j=1}^n\P(G_j\mid\F_{T_j})>m \right\}.$$
Since $\{j\leq\nu_m\}\in\F_{T_j}$, it follows from \eqref{eq:trial-dur} that \begin{align*} \E\left[ \sum_{j=1}^{n\wedge\nu_m} \frac{\tau_j-T_j} {(1+T_j)^{\frac{2\beta}{1+\alpha}}} \right]&\leq C\E\left[ \sum_{j=1}^{n\wedge\nu_m} \left( (1+T_j)^{-\frac{2\beta}{1+\alpha}} + \P(G_j\mid\F_{T_j}) \right) \right]\leq C(m+1). \end{align*}
Moreover, since $T_j\leq S_j$, one has $\{j\leq\nu_m\}\in\F_{S_j}$. Hence, by \eqref{eq:wait},
\begin{align*} &\E\left[ \sum_{j=1}^{n\wedge\nu_m} \frac{T_{j+1}-S_j} {(1+T_j)^{\frac{2\beta}{1+\alpha}}} \right]\leq C\E\left[ \sum_{j=1}^{n\wedge\nu_m} (1+T_j)^{-\frac{2\beta}{1+\alpha}} \right]\leq C(m+1). \end{align*}
Taking $n\to\infty$ and using the monotone convergence theorem, we obtain that a.s. \begin{equation}\label{eq:normalized} \sum_{j=1}^{\nu_m} \frac{ (\tau_j-T_j)+(T_{j+1}-S_j) }{ (1+T_j)^{\frac{2\beta}{1+\alpha}} } < \infty. \end{equation}
By the definition of $\mathcal E$, we note that $\mathcal E \subseteq \bigcup_{m=1}^{\infty}\{\nu_m=\infty\}$. Therefore, by \eqref{eq:normalized},
$$\sum_{j=1}^{\infty} \frac{ (\tau_j-T_j)+(T_{j+1}-S_j) }{ (1+T_j)^{\frac{2\beta}{1+\alpha}} } < \infty\quad\text{a.s. on $\mathcal E$}.$$
Since $0<\frac{\beta}{1+\alpha}<\frac12$, we notice that \begin{align*} &(1+T_{j+1})^{1-\frac{2\beta}{1+\alpha}}-(1+T_j)^{1-\frac{2\beta}{1+\alpha}} \leq \left(1-\frac{2\beta}{1+\alpha}\right)(1+T_j)^{-\frac{2\beta}{1+\alpha}}(T_{j+1}-T_j) \end{align*} almost surely on $\mathcal E$.
By \eqref{eq:Stau}, we note that $T_{j+1}-T_j=(\tau_j-T_j)+(T_{j+1}-S_j)$ for every sufficiently large $j$ almost surely on $\mathcal E$.
Consequently, it follows from \eqref{eq:normalized} that
$$ \sup_{j\geq1}(1+T_j)^{1-\frac{2\beta}{1+\alpha}}<\infty \quad\text{a.s. on $\mathcal E$.}$$
Hence $\sup_{j\geq1}T_j<\infty$ almost surely on $\mathcal E$. This contradicts $T_{j+1}\geq T_j+1$ for every $j$ almost surely on $D^c$, and therefore almost surely on $\mathcal E$. Thus $\P(\mathcal E)=0$, which proves \eqref{eq:diverge}.

\textit{Step 4.} In this step we complete the proof of transience. Choose a deterministic $\varepsilon\in(0,1)$ sufficiently small such that $(1+\varepsilon)^{\frac{\beta}{1+\alpha}}\leq2.$ We take $T_0$ sufficiently large that $\varepsilon(1+T_0)\geq2$.
Define
$$ \widehat G_j := G_j \cap \left\{ T_j<\infty,\ \tau_j-T_j \leq \lfloor\varepsilon(1+T_j)\rfloor \right\}.$$
By the conditional Markov inequality and \eqref{eq:trial-dur}, \begin{align*}
\P(G_j\mid\F_{T_j})-\P(\widehat G_j\mid\F_{T_j}) &\leq \frac{C\left(1+(1+T_j)^{{\frac{2\beta}{1+\alpha}}}\P(G_j\mid\F_{T_j})\right)}{\lfloor\varepsilon(1+T_j)\rfloor}\quad\text{a.s. on $D^c$}.
\end{align*}
Using \eqref{eq:trial-prob} and the fact that $\frac{\beta}{1+\alpha}\in[0,1/2)$, and choosing $T_0$ sufficiently large, we obtain
$\P(\widehat G_j\mid\F_{T_j})\geq\frac12\P(G_j\mid\F_{T_j})$ almost surely on $D^c$.
Therefore,
\begin{equation}\label{eq:timely-prob} \sum_{j=1}^{\infty}\P(\widehat G_j\mid\F_{T_j})=\infty \quad\text{a.s. on $D^c$}.
\end{equation}
Since $\widehat G_j\in\F_{\tau_j}\subseteq\F_{T_{j+1}}$, by the conditional Borel--Cantelli lemma, the events $(\widehat G_j)_{j\geq1}$ occur infinitely often almost surely on $D^c$.
On $\widehat G_j\cap D^c$, we have 
$$ X_{\tau_j}>2R(1+T_j)^{\frac{\beta}{1+\alpha}}\quad\text{and}\quad 1+\tau_j\leq(1+\varepsilon)(1+T_j)$$
and thus
$$ X_{\tau_j}>R(1+\tau_j)^{\frac{\beta}{1+\alpha}}. $$
By Lemma~\ref{lem:escape-scale}, we thus have
$\P(D\mid\F_{\tau_j})\geq1/2$ almost surely on $\widehat G_j\cap D^c$.
By L\'evy's upward theorem, we have $\P(D\mid\F_t)\to\1_D=0$ as $t\to\infty$ a.s. on $D^c$. Since $\widehat G_j$ occurs infinitely often on $D^c$, and the corresponding times $\tau_j$ tend to infinity, this is a contradiction. Hence $X_t\to\infty$ almost surely.
\end{proof}
\section*{Acknowledgements}
Tuan-Minh Nguyen was partially supported by the Australian Research Council under grant ARC DP230102209.

\bibliographystyle{amsplain}
\bibliography{refs}

\end{document}